\documentclass[11pt,reqno]{amsart}
 \usepackage{amssymb,amsfonts,amscd,amsbsy}
\usepackage[mathscr]{eucal}
\usepackage{url}

\newtheorem{thm}{Theorem}[section]
\newtheorem{lem}[thm]{Lemma}
\newtheorem{prop}[thm]{Proposition}
\newtheorem{cor}[thm]{Corollary}
\theoremstyle{definition}

\numberwithin{equation}{section}

\makeatletter
\def\imod#1{\allowbreak\mkern5mu({\operator@font mod}\,\,#1)}
\makeatother

\begin{document}

\title[The Bailey chain and mock theta functions]
{The Bailey chain and mock theta functions} 
 
\author{Jeremy Lovejoy and Robert Osburn}

\address{CNRS, LIAFA, Universit{\'e} Denis Diderot - Paris 7, Case 7014, 75205 Paris Cedex 13, FRANCE}

\address{School of Mathematical Sciences, University College Dublin, Belfield, Dublin 4, Ireland}

\address{IH{\'E}S, Le Bois-Marie, 35, route de Chartres, F-91440 Bures-sur-Yvette, FRANCE}

\email{lovejoy@liafa.jussieu.fr}

\email{robert.osburn@ucd.ie, osburn@ihes.fr}


\date{\today}

\begin{abstract}
Standard applications of the Bailey chain preserve mixed mock modularity but not mock modularity.   After illustrating this with some examples, we show how to use a change of base in Bailey pairs due to Bressoud, Ismail and Stanton to explicitly construct families of $q$-hypergeometric multisums which are mock theta functions.  We also prove identities involving some of these multisums and certain classical mock theta functions.
\end{abstract}

\maketitle

\section{Introduction}

In his plenary address at the Millennial Conference on Number Theory on May 22, 2000, George Andrews challenged mathematicians in the 21th century to elucidate the overlap between classes of $q$-series and modular forms. This challenge has its origin in Ramanujan's last letter to G. H. Hardy on January 12, 1920. In this letter, he introduces 17 ``mock theta functions" such as

\begin{equation} \label{F}
\mathcal{F}_1(q) := \sum_{n \geq 1} \frac{q^{n^2}}{(q^n)_n}.
\end{equation}

\noindent Here, we have used the standard $q$-hypergeometric notation,

\begin{equation*}
(a)_n = (a;q)_n = (1-a)(1-aq)\cdots(1-aq^{n-1}).
\end{equation*}

\noindent Between the time of Ramanujan's death in 1920 and the early part of the $21$st century, approximately 35 other $q$-series were studied and deemed mock theta functions.  Some were introduced by Watson \cite{Wa1}, some were found in Ramanujan's lost notebook and studied by Andrews, Choi, and Hickerson \cite{An-Hi1,Ch1,Ch2,Ch3,Ch4,Ch5}, and others were produced by Berndt, Chan, Gordon and McIntosh using intuition from $q$-series \cite{bc,Go-Mc1,Mc1}.   For a summary of this classical work, see \cite{Go-Mc2} or \cite{Hi-Mo1}. 

Thanks to work of Zwegers and Bringmann and Ono, we now know that each of Ramanujan's original 17 (and the subsequent) examples of mock theta functions is the holomorphic part of a weight $1/2$ harmonic weak Maass form with a weight $3/2$ unary theta function as its ``shadow". Following Zagier, the holomorphic part of any weight $k$ harmonic weak Maass form is called a mock modular form of weight $k$.  It is called a mock theta function if $k=1/2$ and the shadow is a unary theta function.  For more on these functions, their remarkable history and modern developments, see \cite{On1} and \cite{Za1}. 

Returning to Andrews' challenge, a natural question is whether or not there exist other examples of $q$-hypergeometric series which are mock theta functions (in the modern sense).  Several authors have recently addressed this question, constructing two-variable $q$-series which are essentially ``mock Jacobi forms" and which then specialize at torsion points to mock theta functions.  See \cite{Al-Br-Lo1,Br-Lo1,Br-On1,Br-On-Rh1,Ka1,Za1}, for example.  In this paper we investigate the mock modularity of $q$-hypergeometric multisums constructed using the Bailey chain.  

We briefly review Bailey pairs and the Bailey chain.   In the $1940$'s and $50$'s, Bailey and Slater made extensive use of the the fact that if $(\alpha_n,\beta_n)_{n \geq 0}$ 
is a pair of sequences satisfying

\begin{equation} \label{pairdef}
\beta_n = \sum_{k=0}^n \frac{\alpha_k}{(q)_{n-k}(aq)_{n+k}}, 
\end{equation} 

\noindent then subject to convergence conditions we have the identity 

\begin{equation} \label{limitBailey}
\sum_{n \geq 0} (b)_n(c)_n (aq/bc)^n \beta_n = \frac{(aq/b)_{\infty}(aq/c)_{\infty}}{(aq)_{\infty}(aq/bc)_{\infty}} \sum_{n \geq 0} \frac{(b)_n(c)_n(aq/bc)^n }{(aq/b)_n(aq/c)_n}\alpha_n,
\end{equation} 

\noindent where

\begin{equation*}
(a)_{\infty} = (a;q)_{\infty} = \prod_{i=0}^{\infty} (1 - a q^{i}).
\end{equation*}

\noindent For example, Slater \cite{Sl1} collected a long list of pairs satisfying \eqref{pairdef} and a corresponding compendium \cite{Sl2} of $130$ identities of the Rogers-Ramanujan type.  Such identities are best exemplified by the Rogers-Ramanujan identities themselves, which state that for $s=0$ or $1$ we have

\begin{equation} \label{rr}
\sum_{n \geq 0} \frac{q^{n^2+sn}}{(q)_n} = \frac{1}{(q^{1+s};q^5)_{\infty}(q^{4-s};q^5)_{\infty}}.
\end{equation}
In other words, we have a $q$-hypergeometric series expressed as a modular function.

In the $1980$'s Andrews observed that Bailey's work actually leads to a mechanism for producing new pairs satisfying \eqref{pairdef} from known ones \cite{An1,An2}. He called a pair of sequences $(\alpha_n,\beta_n)_{n \geq 0}$ satisfying \eqref{pairdef} a Bailey pair relative to $a$ and showed that if $(\alpha_n,\beta_n)$ is such a sequence, then so is $(\alpha'_n,\beta'_n)$, where

\begin{equation} \label{alphaprimedef}
\alpha'_n = \frac{(b)_n(c)_n(aq/bc)^n}{(aq/b)_n(aq/c)_n}\alpha_n
\end{equation} 

\noindent and

\begin{equation} \label{betaprimedef}
\beta'_n = \sum_{k=0}^n\frac{(b)_k(c)_k(aq/bc)_{n-k} (aq/bc)^k}{(aq/b)_n(aq/c)_n(q)_{n-k}} \beta_k.
\end{equation}
Iterating \eqref{alphaprimedef} and \eqref{betaprimedef} leads to a sequence of Bailey pairs, called the Bailey chain.  

To give an illustration, we follow Chapter $3$ of \cite{An2}.  First, take the so-called unit Bailey pair relative to $a$,

\begin{equation} \label{unit1}
\alpha_n = \frac{(a)_n(1-aq^{2n})(-1)^nq^{\binom{n}{2}}}{(q)_n(1-a)}
\end{equation} 

\noindent and 

\begin{equation} \label{unit2}
\beta_n = \chi(n=0).
\end{equation}

\noindent Then, setting $a=1$ and iterating along the Bailey chain with $b,c \to \infty$ at each step, we arrive at the following generalization of the $s=0$ case of the Rogers-Ramanujan identities \eqref{rr}:

\begin{equation*} \label{Andrews-Gordon}
\begin{aligned}
\sum_{n_{k-1} \geq n_{k-2} \geq \cdots \geq n_1 \geq 0} \frac{q^{n_{k-1}^2 + n_{k-2}^2 + \cdots + n_1^2}}{(q)_{n_{k-1}-n_{k-2}} \cdots (q)_{n_2-n_1}(q)_{n_1}} &= \frac{1}{(q)_{\infty}} \sum_{n \in \mathbb{Z}} (-1)^nq^{kn^2+ \binom{n+1}{2}} \\
&= \frac{(q^k;q^{2k+1})_{\infty}(q^{k+1};q^{2k+1})_{\infty}(q^{2k+1};q^{2k+1})_{\infty}}{(q)_{\infty}}, 
\end{aligned}
\end{equation*} 

\noindent the last equality following from the triple product identity, 

\begin{equation*} \label{jtp}
\sum_{n \in \mathbb{Z}} z^nq^{n^2} = (-zq;q^2)_{\infty}(-q/z;q^2)_{\infty}(q^2;q^2)_{\infty}.
\end{equation*}

The point is that iteration along the Bailey chain preserves the number-theoretic structure on the $\alpha$-side, and now instead of each Bailey pair giving rise to a single modular $q$-hypergeometric series, each pair leads to a family of modular $q$-hypergeometric multisums.   As a bonus, these multisums naturally occur in many areas of mathematics.   For references to the role of such series in combinatorics, statistical mechanics, Lie algebras, and group theory, see \cite{Fu1}.  For novel interactions with knot theory, see \cite{Ar-Da1} and \cite{gl}.

Now consider an example involving mock theta functions. Take the Bailey pair

\begin{equation} \label{alphaf}
\alpha_n = 
\begin{cases}
1, & \text{if $n=0$},\\
\frac{4(-1)^nq^{\binom{n+1}{2}}}{(1+q^n)}, & \text{otherwise},
\end{cases}
\end{equation}  

\noindent and

\begin{equation} \label{betaf}
\beta_n  = \frac{1}{(-q)_n^2},
\end{equation}

\noindent which follows directly upon substituting \eqref{unit1} and \eqref{unit2} into \eqref{alphaprimedef} and \eqref{betaprimedef} with $-a=b=c=-1$. Iteration along the Bailey chain with $b,c \to \infty$ at each step gives

\begin{equation} \label{generalizedfofq}
\sum_{n_k \geq n_{k-1} \geq \cdots \geq n_1 \geq 0} \frac{q^{n_k^2+n_{k-1}^2+\cdots+n_1^2}}{(q)_{n_k-n_{k-1}}\cdots(q)_{n_2-n_1}(-q)_{n_1}^2} = \frac{2}{(q)_{\infty}}\sum_{n \in \mathbb{Z}} \frac{q^{kn^2+\binom{n+1}{2}}(-1)^n}{(1+q^n)}.
\end{equation}  

\noindent The case $k=1$ of \eqref{generalizedfofq} is Watson's expression for Ramanujan's third order mock theta function $f(q)$ as an Appell-Lerch sum \cite{Wa1}.  For general $k$ the left-hand side may be interpreted as a generating function for partitions weighted according to certain ranks \cite{Ga1}.   However, the sums on the right-hand side of \eqref{generalizedfofq} are known as ``higher level" Appell functions \cite{Zw2,Zw3} and in general give rise not to mock but to \emph{mixed} mock modular forms, that is, sums of the form $\sum_{i=1}^n f_ig_i$, where $f_i$ is modular and $g_i$ is mock modular.  In other words, it appears that standard applications of the Bailey chain preserve the space of mixed mock modular forms, but typically fail to produce families of mock theta functions.  In Section $3$, we discuss \eqref{generalizedfofq} and another example in detail.

Now, mixed mock modular forms are certainly interesting and important functions.  They have recently appeared as characters arising from affine Lie superalgebras \cite{BO1}, as generating functions for exact formulas for the Euler numbers of certain moduli spaces \cite{bm}, as generating functions for Joyce invariants \cite{Me-Ok1}, in the quantum theory of black holes and wall-crossing phenomenon \cite{dmz}, in relation to other automorphic objects \cite{bk,cr}, and in the combinatorial setting of $q$-series and partitions (e.g. \cite{andrewsshort,andrewsrep,arz,bhmv,bma,rhoades}). 

But what about the special structure of ``pure" mock modular forms?   We shall observe that it is possible to preserve the mock modularity in $q$-hypergeometric multisums constructed using the Bailey machinery, and one way to do so is to appeal to change of base formulas like those in \cite{Be-Wa1,bis,St1}.  We make use of the following change of base due to Bressoud, Ismail, and Stanton:

\begin{lem}{\cite[Theorem 2.5, $a=q$ and $B \to \infty$]{bis}} If $(\alpha_n,\beta_n)$ is a Bailey pair relative to $q$, so is $(\alpha'_n,\beta'_n)$ where

\begin{equation}  \label{alphabase}
\alpha'_n = \frac{(1+q)}{(1+q^{2n+1})}q^n\alpha_n(q^2) 
\end{equation}

\noindent and

\begin{equation} \label{betabase}
\beta'_n = \sum_{k = 0}^n \frac{(-q)_{2k}q^k}{(q^2;q^2)_{n-k}}\beta_k(q^2).
\end{equation}            
\end{lem}  

We present four examples in our main result.

\begin{thm} \label{main} 
Write

\begin{equation*} \label{Bn}
\begin{aligned}
B_k&(n_k,n_{k-1},\dots\,n_1;q) := q^{\binom{n_{k-1}+1}{2} + n_{k-2} + 2n_{k-3} + \cdots + 2^{k-3}n_1} (-1)^{n_1} \\
&\times \frac{(-q)_{n_{k-1}}(-q)_{2n_{k-2}}(-q^2;q^2)_{2n_{k-3}}\cdots(-q^{2^{k-3}};q^{2^{k-3}})_{2n_1}}{(q)_{n_k-n_{k-1}}(q^2;q^2)_{n_{k-1}-n_{k-2}} \cdots (q^{2^{k-2}};q^{2^{k-2}})_{n_2-n_1}(q^{2^{k-1}};q^{2^{k-1}})_{n_1}}.
\end{aligned}
\end{equation*}

\noindent For $k \geq 3$ the following are mock theta functions:

\begin{eqnarray} 
\mathcal{R}_1^{(k)}(q) &:=& \sum_{n_k \geq n_{k-1} \geq \dotsc \geq n_1 \geq 0} 
q^{\binom{n_k+1}{2}} B_k(n_k,\dots,n_1;q), \label{main1} \\
\mathcal{R}_2^{(k)}(q) &:=& \sum_{n_k \geq n_{k-1} \geq \dotsc \geq n_1 \geq 0}  
\frac{q^{n_k^2 + n_k}}{(-q)_{n_k}}B_k(n_k,\dots,n_1;q), \label{main2} \\
\mathcal{R}_3^{(k)}(q) &:=& \sum_{n_k \geq n_{k-1} \geq \dotsc \geq n_1 \geq 0}  
\frac{q^{n_k^2+2n_k} (-1)^{n_k}(q;q^2)_{n_k}}{(-q^2;q^2)_{n_k}}B_k(n_k,\dots,n_1;q^2), \label{main3} \\
\mathcal{R}_4^{(k)}(q) &:=& \sum_{n_k \geq n_{k-1} \geq \dotsc \geq n_1 \geq 0}  \frac{q^{n_k} (-1)^{n_k}(q;q^2)_{n_k}}{(-q)_{n_k}}B_k(n_k,\dots,n_1;q). \label{main4}
\end{eqnarray}
\end{thm}

In order to satisfy the claim that the above series are mock theta functions, we will express them, up to the addition of weakly holomorphic modular forms, as specializations of Appell-Lerch sums $m(x,q,z)$, where

\begin{equation*} 
m(x,q,z) := \frac{1}{j(z,q)} \sum_{r \in \mathbb{Z}} \frac{(-1)^r q^{\binom{r}{2}} z^r}{1-q^{r-1} xz}.
\end{equation*}

\noindent Here $x$, $z \in \mathbb{C}^{*}:=\mathbb{C} \setminus \{ 0 \}$ with neither $z$ nor $xz$ an integral power of $q$, and 

$$
j(x,q):=(x)_{\infty} (q/x)_{\infty} (q)_{\infty}.
$$

\noindent The fact that specializations of Appell-Lerch sums give mock theta functions essentially follows from work of Zwegers \cite[Ch. 1]{Zw1}.   We note that although Zwegers' results are expressed in terms of his mock Jacobi form $\mu(u,v,\tau)$, one can easily translate $\mu(u,v,\tau)$ into $m(x, q, z)$ and vice versa. 

We proceed by first using the Bailey machinery to express the multisums in terms of indefinite theta series $f_{a,b,c}(x,y,q)$, where

\begin{equation} \label{fdef}
f_{a,b,c}(x,y,q) : = \sum_{\text{sg}(r) = \text{sg}(s)} \text{sg}(r) (-1)^{r+s} x^r y^s q^{a \binom{r}{2} + brs + c \binom{s}{2}}.
\end{equation}

\noindent Here $x$, $y \in \mathbb{C}^{*}$ and sg$(r):=1$ for $r \geq 0$ and sg$(r):=-1$ for $r <0$.  One could then follow Chapter 2 of \cite{Zw1}, but instead we apply recent results of Hickerson and Mortenson \cite{Hi-Mo1} to convert the indefinite theta series to Appell-Lerch series (see equations \eqref{casei}, \eqref{caseii}, \eqref{caseiibis}, \eqref{caseiii}, and \eqref{caseiv}). Some background material on indefinite theta functions and Appell-Lerch series is collected in Section 2, and Theorem \ref{main} is established in Section 4.   

In Section 5, we prove identities between some of the multisums in Theorem \ref{main} and some of the classical $q$-hypergeometric mock theta functions.  Recall \eqref{F} as well as the mock theta functions $\nu(q)$, $\phi(q)$, and $\mu(q)$ (historically referred to as having ``orders" $7$, $3$, $10$, and $2$, respectively):

\begin{equation} \label{nu}
\nu(q) := \sum_{n \geq 0} \frac{q^{n^2+n}}{(-q;q^2)_{n+1}},
\end{equation} 

\begin{equation} \label{phi}
\phi(q) := \sum_{n \geq 0} \frac{q^{\binom{n+1}{2}}}{(q;q^2)_{n+1}},
\end{equation}   

\noindent and

\begin{equation} \label{mu}
\mu(q) := \sum_{n \geq 0} \frac{(-1)^nq^{n^2}(q;q^2)_n}{(-q^2;q^2)_n^2}.
\end{equation}

\begin{cor} \label{mockid} We have the following identities.

\begin{equation} \label{mockid1}
\mathcal{R}_1^{(3)}(q) = \nu(-q),
\end{equation}

\begin{equation} \label{mockid2}
\mathcal{R}_1^{(4)}(q)  = -\phi(q^4) + M_1(q),
\end{equation}

\begin{equation} \label{mockid3}
\mathcal{R}_2^{(3)}(q) = q^{-1} \mathcal{F}_{1}(q^4) + M_2(q),
\end{equation}

\begin{equation} \label{mockid5}
\mathcal{R}_4^{(k)}(q) = q^{-2^{k-3}(2^{k-2} + 1)} \mu(q^{2^{k-1}(2^{k-1} + 1)}) + M_3^{(k)}(q),
\end{equation}

\noindent where $M_1(q)$, $M_2(q)$, and $M_3^{(k)}(q)$ are (explicit) weakly holomorphic modular forms.
\end{cor}

\section{Indefinite theta series and Appell-Lerch series}
We recall some facts from \cite{Hi-Mo1}.   The most important of these is a result which allows us to convert from indefinite theta series \eqref{fdef} to Appell-Lerch series. Define

\begin{equation} \label{g}
\begin{aligned}
g_{a,b,c}(x, y, q, z_1, z_0) & := \sum_{t=0}^{a-1} (-y)^t q^{c\binom{t}{2}} j(q^{bt} x, q^a) m\left(-q^{a \binom{b+1}{2} - c \binom{a+1}{2} - t(b^2 - ac)} \frac{(-y)^a}{(-x)^b}, q^{a(b^2 - ac)}, z_0 \right)\\
& + \sum_{t=0}^{c-1} (-x)^t q^{a \binom{t}{2}} j(q^{bt} y, q^c) m\left(-q^{c\binom{b+1}{2} - a\binom{c+1}{2} - t(b^2 -ac)}  \frac{(-x)^c}{(-y)^b}, q^{c(b^2 - ac)}, z_1\right)
\end{aligned}
\end{equation}

\noindent and

\begin{equation*}
\begin{aligned}
& \theta_{n,p}(x,y,q) := \frac{1}{\overline{J}_{0, np(2n+p)}} \sum_{r^{*} = 0}^{p-1} \sum_{s^{*}=0}^{p-1} q^{n\binom{r-(n-1)/2}{2} + (n+p)(r - (n-1)/2)(s+ (n+1)/2) + n \binom{s + (n+1)/2}{2}} \\
& \times \frac{(-x)^{r - (n-1)/2} (-y)^{s + (n+1)/2} J_{p^{2} (2n+p)}^{3} j(-q^{np(s-r)} x^{n} / y^{n}, q^{np^2}) j(q^{p(2n+p)(r+s) + p(n+p)} x^{p} y^{p}, q^{p^2 (2n + p)})}{j(q^{p(2n+p)r + p(n+p)/2} (-y)^{n+p} / (-x)^{n}, q^{p^{2} (2n+p)}) j(q^{p(2n+p)s + p(n+p)/2} (-x)^{n+p} / (-y)^{n}, q^{p^{2} (2n+p)})},
\end{aligned}
\end{equation*}

\noindent where $r := r^{*} + \{(n-1)/2 \}$ and $s:= s^{*} + \{ (n-1)/2 \}$ with $0 \leq \{ \alpha \} < 1$ denoting the fractional part of $\alpha$.  Also, $J_{m}:= J_{m, 3m}$ with $J_{a,m} := j(q^{a}, q^{m})$, and $\overline{J}_{a,m}:=j(-q^{a}, q^{m})$.

Following \cite{Hi-Mo1}, we use the term ``generic" to mean that the parameters do not cause poles in the Appell-Lerch sums or in the quotients of theta functions.

\begin{thm}{\cite[Theorem 0.3]{Hi-Mo1}} \label{hm} Let $n$ and $p$ be positive integers with $(n$, $p)=1$. For generic $x$, $y \in \mathbb{C}^{*}$

$$
f_{n, n+p, n}(x, y, q) = g_{n, n+p, n}(x, y, q, -1, -1) + \theta_{n,p}(x,y,q).
$$

\end{thm}

We shall also require certain facts about $j(x,q)$, $m(x,q,z)$, and $f_{a,b,c}(x,y,q)$, which we collect here.   First,
from the definition of $j(x,q)$, we have

\begin{equation} \label{j1}
j(q^{n} x, q) = (-1)^{n} q^{-\binom{n}{2}} x^{-n} j(x,q)
\end{equation}

\noindent where $n \in \mathbb{Z}$ and

\begin{equation} \label{j2}
j(x,q) = j(q/x, q) = -x j(x^{-1}, q).
\end{equation}

Next, a relevant property of the sum $m(x, q, z)$ is given in the following (see Corollary 3.11 in \cite{mort}). 

\begin{prop} \label{mprops} For generic $x$, $z \in \mathbb{C}^{*}$

\begin{equation} \label{m3}
m(x, q, z) = m(-qx^2, q^4, -1) - q^{-1} x m(-q^{-1}x^2, q^4, -1) - \xi(x,q,z)
\end{equation}

\noindent where 

$$
\xi(x,q,z):= \frac{J_2^3}{j(xz, q) j(qx^2, q^2) \overline{J}_{0,4}} \Biggl[ \frac{j(qx^{2} z, q^2) j(-z^2, q^4)}{j(z,q^2)} - xz \frac{j(q^{2} x^{2} z, q^2) j(-q^{2} x^{2}, q^4)}{j(qz, q^2)} \Biggr].
$$

\end{prop}

Finally, two important transformation properties of $f_{a,b,c}(x,y,q)$ are given in the following (see Propositions 5.1 and 5.2 in \cite{Hi-Mo1}).

\begin{prop} For $x$, $y \in \mathbb{C}^{*}$,

\begin{equation} \label{fprop1}
\begin{aligned}
f_{a,b,c}(x,y,q) = & f_{a,b,c}(-x^2 q^a, -y^2 q^c, q^4) - x f_{a,b,c}(-x^2 q^{3a}, -y^2 q^{c+2b}, q^4) \\
& -y f_{a,b,c}(-x^2 q^{a+2b}, -y^2 q^{3c}, q^4) + xyq^b f_{a,b,c}(-x^2 q^{3a+2b}, -y^2 q^{3c+2b}, q^4)
\end{aligned}
\end{equation}

\noindent and 

\begin{equation} \label{fprop2}
f_{a,b,c}(x, y, q)= - \frac{q^{a+b+c}}{xy} f_{a,b,c}(q^{2a+b}/ x, q^{2c+b} / y, q).
\end{equation}

\end{prop}

\section{The Bailey chain and mixed mock modular forms}
All known $q$-hypergeometric mock theta functions are expressible via the Bailey lemma in terms of Appell-Lerch series and/or indefinite theta functions.   Iterating the relevant Bailey pairs using \eqref{alphaprimedef} and \eqref{betaprimedef} provides a virtually endless source of mixed mock modular forms.   

To illustrate what happens in the case of Appell-Lerch series, let us return to the example from the introduction.  Let $\mathcal{B}_{1}^{(k)}(q)$ denote the multisum appearing in \eqref{generalizedfofq}.  

Using the identities (for generic $x$ and $q$)

$$
\frac{1}{1-x} = \frac{1+x + \cdots + x^{2k}}{1-x^{2k+1}}
$$

\noindent and

$$
\sum_{n \in \mathbb{Z}} \frac{(-1)^n q^{\binom{n+1}{2}}}{1-xq^n} = \frac{(q)_{\infty}^2}{(x)_{\infty} (q/x)_{\infty}},
$$

\noindent equation \eqref{generalizedfofq} gives

\begin{equation*}
\mathcal{B}_{1}^{(k)}(q) = \frac{2}{(q)_{\infty}} \Biggl(  \sum_{\substack{i=1 \\ i \neq k+1}}^{2k+1} (-1)^{i+1} j(q^{k+i}, q^{2k+1}) m(-q^{k-i+1}, q^{2k+1}, q^{k+i}) + (-1)^k\frac{(q^{2k+1}; q^{2k+1})_{\infty}^2}{2(-q^{2k+1}; q^{2k+1})_{\infty}^2} \Biggr).
\end{equation*}
Since the specialized Appell-Lerch series $m(x,q,z)$ is generically a mock theta function, $\mathcal{B}_{1}^{(k)}(q)$ is in general a mixed mock modular form.   When $k=1$ we have $j(q^2,q^3) = -qj(q^4,q^3) = (q)_{\infty}$ and so one ``accidentally" obtains a genuine mock theta function.

The case of $\mathcal{B}_{1}^{(k)}(q)$ is typical.  Iteration along the Bailey chain produces series of the form $\frac{1}{f}A_{\ell}(a,b,q)$, where $f$ is a modular form and $A_{\ell}(a,b,q)$ is the ``level $\ell$" Appell sum \cite{Zw2},
\begin{equation}
A_{\ell}(a,b,q) := a^{\ell/2}\sum_{n \in \mathbb{Z}} \frac{(-1)^{\ell n}q^{\ell n(n+1)/2}b^n}{1-aq^n}.
\end{equation} 
It is known that the $A_{\ell}(a,b,q)$ are generically mixed mock modular forms \cite{Zw2}.

To illustrate what happens in the case of indefinite theta functions, consider \eqref{phi}, which satisfies \cite{Ch1}
\begin{equation*}
\begin{aligned}
\phi(q) & = \frac{(-q)_{\infty}}{(q)_{\infty}}\Bigl( \sum_{r, s \geq 0} - \sum_{r, s < 0} \Bigr) (-1)^{r+s}q^{r^2 + r + 3rs + s^2 + s} \\
& =  \frac{(-q)_{\infty}}{(q)_{\infty}} f_{2,3,2}(q^2,q^2,q). 
\end{aligned}
\end{equation*}  
Using the relevant Bailey pair (see \cite{Ch1}), iterating along the Bailey chain with $b,c \to \infty$ at each step, and then substituting into \eqref{limitBailey} with $b =-q$ and $c \to \infty$, we obtain

\begin{equation*}
\begin{aligned}
\mathcal{B}_{2}^{(k)}(q) & := \sum_{n_k \geq n_{k-1} \geq \cdots \geq n_1 \geq  0} \frac{(-q)_{n_k} q^{\binom{n_k + 1}{2} + n_{k-1}^2 + n_{k-1} + \cdots + n_1^2 + n_1}}{(q)_{n_k - n_{k-1}} \cdots (q)_{n_2 - n_1} (q^{n_1 + 1})_{n_1 + 1}} \\
& = \frac{(-q)_{\infty}}{(q)_{\infty}} \Bigl( \sum_{r, s \geq 0} - \sum_{r, s < 0} \Bigr) (-1)^{r+s}q^{kr^2 + kr + (2k+1)rs + ks^2 + ks} \\
& =  \frac{(-q)_{\infty}}{(q)_{\infty}} f_{2k, 2k+1, 2k}(q^{2k}, q^{2k}, q). 
\end{aligned}
\end{equation*}
The fact that we have a genuine mock theta function for $k=1$ is an accident.  Theorem \ref{hm} clearly shows that these are, in general, mixed mock theta functions.   This is typical for indefinite theta functions.

\section{Proof of Theorem \ref{main} }
We begin by establishing our key Bailey pair.

\begin{prop} \label{keyprop}
The sequences $\alpha_{n_k}$ and $\beta_{n_k}$ form a Bailey pair relative to $q$, where

\begin{equation*} \label{keyalpha}
\alpha_{n_k} = \frac{q^{((2^{k-1}+1)n_k^2+(2^{k-1}-1)n_k)/2}(1-q^{2n_k+1})}{(1-q)}\sum_{|j| \leq n_k} (-1)^jq^{-2^{k-2}j^2}
\end{equation*}

\noindent and

\begin{equation*} \label{keybeta}
\beta_{n_k} = \frac{1}{(-q)_{n_k}}\sum_{n_k \geq n_{k-1} \cdots \geq n_1 \geq 0} B_k(n_k,n_{k-1},\dots,n_1;q).
\end{equation*}
\end{prop}

\begin{proof}

Consider the Bailey pair relative to $q$,

\begin{equation*}
\alpha_n = \frac{q^{n^2}(1-q^{2n+1})}{(1-q)}\sum_{|j| \leq n} (-1)^jq^{-j^2}
\end{equation*}
and
\begin{equation*}
\beta_n = \frac{(-1)^n}{(q^2;q^2)_n}.
\end{equation*}
This may be read off from the case $(a,b,c) \to (q,-1,0)$ of Theorem 2.2 of \cite{An-Hi1}. Iterating using \eqref{alphabase} and \eqref{betabase} gives two sequences,

\begin{equation*}
\begin{aligned}
\alpha_n &= \frac{q^{n^2}(1-q^{2n+1})}{(1-q)}\sum_{|j| \leq n} (-1)^jq^{-j^2}, \\
\alpha'_n &= \frac{q^{2n^2+n}(1-q^{2n+1})}{(1-q)}\sum_{|j| \leq n} (-1)^jq^{-2j^2}, \\
\alpha''_n &= \frac{q^{4n^2+3n}(1-q^{2n+1})}{(1-q)}\sum_{|j| \leq n} (-1)^jq^{-4j^2}, \\
\alpha'''_n &= \frac{q^{8n^2+7n}(1-q^{2n+1})}{(1-q)}\sum_{|j| \leq n} (-1)^jq^{-8j^2}, \\
&\vdots
\end{aligned}
\end{equation*}
and
\begin{equation*}
\begin{aligned}
\beta_n &= \frac{(-1)^n}{(q^2;q^2)_n}, \\
\beta'_n &= \sum_{n \geq n_1 \geq 0} \frac{(-q)_{2n_1} q^{n_1} (-1)^{n_1}}{(q^2;q^2)_{n-n_1}(q^4;q^4)_{n_1}}, \\
\beta''_n &= \sum_{n \geq n_2 \geq n_1 \geq 0} \frac{(-q)_{2n_2}(-q^2;q^2)_{2n_1}q^{n_2+2n_1}(-1)^{n_1}}{(q^2;q^2)_{n-n_2}(q^4;q^4)_{n_2-n_1}(q^8;q^8)_{n_1}}, \\
\beta'''_n &= \sum_{n \geq n_3 \geq n_2 \geq n_1 \geq 0} \frac{(-q)_{2n_3}(-q^2;q^2)_{2n_2}(-q^4;q^4)_{2n_1}q^{n_3+2n_2+4n_1}(-1)^{n_1}}{(q^2;q^2)_{n-n_3}(q^4;q^4)_{n_3-n_2}(q^8;q^8)_{n_2-n_1}(q^{16};q^{16})_{n_1}}, \\
&\vdots
\end{aligned}
\end{equation*}

\noindent The general terms are 

\begin{equation*}
\alpha^{(k)}_n = \frac{q^{2^kn^2+(2^k-1)n}(1-q^{2n+1})}{(1-q)}\sum_{|j| \leq n} (-1)^jq^{-2^kj^2}
\end{equation*} 
and
\begin{equation*}
\beta^{(k)}_n = \sum_{n \geq n_{k} \geq n_{k-1} \cdots \geq n_1 \geq 0} \frac{(-q)_{2n_k}(-q^2;q^2)_{2n_{k-1}}\cdots(-q^{2^{k-1}};q^{2^{k-1}})_{n_1}q^{n_k+2n_{k-1}+\cdots+2^{k-1}n_1}(-1)^{n_1}}{(q^2;q^2)_{n-n_k}(q^4;q^4)_{n_k-n_{k-1}}\cdots (q^{2^{k}};q^{2^{k}})_{n_2-n_1} (q^{2^{k+1}};q^{2^{k+1}})_{n_1}}.
\end{equation*}
We then apply \eqref{alphaprimedef} and \eqref{betaprimedef} with $b=-q$ and $c \to \infty$, shifting $k \to k-2$ and replacing $n$ by $n_k$ to obtain the result.

\end{proof}

\begin{proof}[Proof of Theorem \ref{main}]

For (\ref{main1}), apply Proposition \ref{keyprop} and let $b=-q$ and $c \to \infty$ in \eqref{limitBailey} to obtain

\begin{equation} \label{first}
\begin{aligned}
\mathcal{R}_1^{(k)}(q) &= \frac{(-q)_{\infty}}{(q)_{\infty}} \sum_{n \geq 0 \atop |j| \leq n} (-1)^j q^{(2^{k-2}+1)n^2 +2^{k-2}n - 2^{k-2}j^2}(1-q^{2n+1}) \\
& =  \frac{(-q)_{\infty}}{(q)_{\infty}} \Biggl( \sum_{n \geq 0} q^{(2^{k-2} + 1)n^2 + {2^{k-2}}n} \sum_{j=-n}^{n} (-1)^{j} q^{-2^{k-2} j^2} \\
& \hskip1in - q \sum_{n \geq 0} q^{(2^{k-2} + 1)n^2 + (2^{k-2} + 2)n} \sum_{j=-n}^{n} (-1)^{j} q^{-2^{k-2} j^2} \Biggr).
\end{aligned}
\end{equation}   

After replacing $n$ with $-n-1$ in the second sum of (\ref{first}), we let $n=(r+s)/2$ and $j=(r-s)/2$ to find

\begin{equation*}
\begin{aligned}
\mathcal{R}_1^{(k)}(q) & = \frac{(-q)_{\infty}}{(q)_{\infty}} \Biggl ( \Bigl(  \sum_{\substack{r, s \geq 0 \\ r \equiv s \imod{2} }} - \sum_{\substack{r, s < 0 \\ r \equiv {s \imod{2}} }} \Bigr) (-1)^{\frac{r-s}{2}} q^{\frac{1}{4} r^2 + \frac{1}{2}(2^{k-1} + 1)rs + 2^{k-3}r + \frac{1}{4}s^2 + 2^{k-3} s}   \Biggr) \\
& = \frac{(-q)_{\infty}}{(q)_{\infty}} \Bigl(f_{1,2^{k-1}+1,1}(q^{2^{k-2}+1},q^{2^{k-2}+1},q^2) \\ 
& \hskip1in + q^{2^{k-1}+1}f_{1,2^{k-1}+1,1}(q^{3(2^{k-2}+1)},q^{3(2^{k-2}+1)},q^2) \Bigr) \\
\end{aligned}
\end{equation*}

\noindent where in the last step we let $r \to 2r$ and $s \to 2s$, then let $r \to 2r+1$ and $s \to 2s+1$ and invoke (\ref{fdef}). By Theorem \ref{hm}, (\ref{g}), (\ref{j1}) and (\ref{j2}), we have

\begin{equation} \label{a1}
\begin{aligned}
 f_{1,2^{k-1}+1,1}(q^{2^{k-2}+1},q^{2^{k-2}+1},q^2) 
&= 2 (-1)^{2^{k-3}} q^{-2^{2k-6}} j(q, q^2) m(-q^{2^k (2^{k-3} + 1)}, q^{2^{k+1} (2^{k-2} + 1)}, -1) \\ &+ \theta_{1, 2^{k-1}}(q^{2^{k-2}+1}, q^{2^{k-2}+1}, q^2)
\end{aligned}
\end{equation}

\noindent and

\begin{equation} \label{a2}
\begin{aligned}
f_{1,2^{k-1}+1,1}&(q^{3(2^{k-2}+1)},q^{3(2^{k-2}+1)},q^2) 
\\ &= -2 (-1)^{2^{k-3}} q^{3 \cdot 2^{k-3} (-3 \cdot 2^{k-3} - 2) - 1} j(q, q^2) m(-q^{-2^{2k-3}}, q^{2^{k+1} (2^{k-2} + 1)}, -1) \\ &+ \theta_{1, 2^{k-1}} (q^{3(2^{k-2}+1)}, q^{3(2^{k-2}+1)}, q^2).
\end{aligned}
\end{equation}

Combining (\ref{a1}) and (\ref{a2}) and applying (\ref{m3}) implies that 

\begin{equation} \label{casei}
\begin{aligned}
\mathcal{R}_1^{(k)}(q) &=  2(-1)^{2^{k-3}} q^{-2^{2k-6}} \bigl [ m(q^{2^{k-2}}, q^{2^{k-1} (2^{k-2} + 1)}, z) + \xi(q^{2^{k-2}}, q^{2^{k-1} (2^{k-2} + 1)}, z)  \bigr] \\
& + \frac{\theta_{1, 2^{k-1}}(q^{2^{k-2}+1},q^{2^{k-2}+1},q^2)}{j(q, q^2)} \\ & + q^{2^{k-1} + 1} \frac{\theta_{1, 2^{k-1}}(q^{3(2^{k-2}+1)},q^{3(2^{k-2}+1)},q^2)}{j(q, q^2)}. 
\end{aligned}
\end{equation}

Next, for (\ref{main2}), apply Proposition \ref{keyprop} and let $b,c \to \infty$ in \eqref{limitBailey} to obtain 

\begin{equation} \label{second}
\begin{aligned}
\mathcal{R}_2^{(k)}(q) =& \frac{1}{(q)_{\infty}} \sum_{n \geq 0 \atop |j| \leq n} (-1)^j q^{((2^{k-1}+3)n^2 +(2^{k-1}+1)n)/2 - 2^{k-2}j^2}(1-q^{2n+1}) \\
& = \frac{1}{(q)_{\infty}} \Biggl ( \sum_{n \geq 0} q^{\frac{2^{k-1} + 3}{2} n^2 + \frac{2^{k-1} + 1}{2} n} \sum_{j=-n}^{n} (-1)^j q^{-2^{k-2} j^2}  \\
& \hskip1in - q \sum_{n \geq 0} q^{\frac{2^{k-1} + 3}{2} n^2 + \frac{2^{k-1} + 5}{2} n} \sum_{j=-n}^{n} (-1)^j q^{-2^{k-2} j^2} \Biggr).
\end{aligned}
\end{equation} 

We again replace $n$ with $-n-1$ in the second sum of (\ref{second}), then let $n=(r+s)/2$ and $j=(r-s)/2$ to get

\begin{equation*}
\begin{aligned}
\mathcal{R}_2^{(k)}(q) & =  \frac{1}{(q)_{\infty}}  \Biggl ( \Bigl(  \sum_{\substack{r, s \geq 0 \\ r \equiv s \imod{2} }} - \sum_{\substack{r, s < 0 \\ r \equiv {s \imod{2}} }} \Bigr) (-1)^{\frac{r-s}{2}} q^{\frac{3}{8} r^2 + \frac{2^k + 3}{4} rs + \frac{2^{k-1} + 1}{4} r + \frac{3}{8} s^2 + \frac{2^{k-1} + 1}{4}s}   \Biggr) \\
& = \frac{1}{(q)_{\infty}}\left(f_{3,2^k+3,3}(q^{2^{k-2}+2},q^{2^{k-2}+2},q) + q^{2^{k-1}+2}f_{3,2^k+3,3}(q^{3(2^{k-2}+2)-1},q^{3(2^{k-2}+2)-1},q)\right).
\end{aligned}
\end{equation*}

\noindent Now, for $k$ odd, apply Theorem \ref{hm} twice and simplify using (\ref{g}), (\ref{j1}) and (\ref{j2}) to get

\begin{equation} \label{second1}
\begin{aligned}
f_{3,2^k+3,3}(q^{2^{k-2}+2},q^{2^{k-2}+2},q) & = -2q^{-\frac{2^{k-3}(2^{k-2}+1)}{3}} j(q, q^3) m(-q^{2^{k-1} (5 \cdot 2^{k-1} + 17)}, q^{3 \cdot 2^{k+1} (2^{k-1} + 3)}, -1) \\
& + 2q^{-35\cdot2^{k-3}-27\cdot2^{2k-5}} j(q, q^3) m(-q^{2^{k-1}(-3 \cdot 2^{k-1} - 7)}, q^{3 \cdot 2^{k+1} (2^{k-1} + 3)}, -1) \\ & + \theta_{3, 2^k}(q^{2^{k-2}+2},q^{2^{k-2}+2},q) 
\end{aligned}
\end{equation}

\noindent and

\begin{equation} \label{second2}
\begin{aligned}
 q^{2^{k-1} + 2} f_{3,2^k+3,3}&(q^{3 \cdot 2^{k-2} + 5}, q^{3 \cdot 2^{k-2} + 5}, q) \\
& = -2q^{-3\cdot 2^{k-3} (2^{k-2} + 1)} j(q, q^3) m(-q^{2^{k-1}(3 \cdot 2^{k-1} + 11)}, q^{3 \cdot 2^{k+1} (2^{k-1} + 3)}, -1) \\
& + 2 q^{\frac{-61 \cdot 2^{k-3} - 49 \cdot 2^{2k-5}}{3}} j(q, q^3)  m(-q^{-2^{k-1}(2^{k-1} + 1)}, q^{3 \cdot 2^{k+1} (2^{k-1} + 3)}, -1) \\
& + q^{2^{k-1} + 2} \theta_{3, 2^k}(q^{3 \cdot 2^{k-2} + 5}, q^{3 \cdot 2^{k-2} + 5}, q).
\end{aligned}
\end{equation}

Combining the first $m$ in (\ref{second1}) with the second $m$ in (\ref{second2}) and applying (\ref{m3}) yields

\begin{equation*}
-2q^{-\frac{2^{k-3}(2^{k-2}+1)}{3}} \bigl [ m(q^{2^{k+1} (2^{k-3} + 1)}, q^{3 \cdot 2^{k-1} (2^{k-1} + 3)}, z) + \xi(q^{2^{k+1} (2^{k-3} + 1)}, q^{3 \cdot 2^{k-1} (2^{k-1} + 3)}, z) \bigr ]
\end{equation*}

\noindent while the first $m$ in (\ref{second2}) with the second $m$ in (\ref{second1}) and (\ref{m3}) gives

\begin{equation*}
-2 q^{-3 \cdot 2^{k-3} (2^{k-2} + 1)} \bigl[ m(q^{2^{k-1}}, q^{3\cdot 2^{k-1} (2^{k-1} + 3)}, z) + \xi(q^{2^{k-1}}, q^{3\cdot 2^{k-1} (2^{k-1} + 3)}, z) \bigr ].
\end{equation*}

\noindent In total, we have

\begin{equation} \label{caseii}
\begin{aligned}
\mathcal{R}_2^{(k)}(q) & = -2q^{\frac{-2^{k-3}(2^{k-2}+1)}{3}} \bigl [ m(q^{2^{k+1}(2^{k-3} + 1)}, q^{3 \cdot 2^{k-1} (2^{k-1} + 3)}, z) 
 + \xi(q^{2^{k+1} (2^{k-3} + 1)}, q^{3 \cdot 2^{k-1} (2^{k-1} + 3)}, z) \bigr ] \\
& -2 q^{-3 \cdot 2^{k-3} (2^{k-2} + 1)} \bigl[ m(q^{2^{k-1}}, q^{3\cdot 2^{k-1} (2^{k-1} + 3)}, z) + \xi(q^{2^{k-1}}, q^{3\cdot 2^{k-1} (2^{k-1} + 3)}, z) \bigr ] \\
& + \frac{\theta_{3, 2^k}(q^{2^{k-2}+2},q^{2^{k-2}+2},q)}{j(q, q^3)} +  \frac{q^{2^{k-1} + 2} \theta_{3, 2^k}(q^{3 \cdot 2^{k-2} + 5}, q^{3 \cdot 2^{k-2} + 5}, q)}{j(q, q^3)}.
\end{aligned}
\end{equation}

\noindent One can similarly show that for $k$ even, we have

\begin{equation*} \label{1}
\begin{aligned}
f_{3, 2^{k}+ 3, 3}(q^{2^{k-2} + 2}, q^{2^{k-2} + 2}, q) 
&= -2 q^{\frac{-29\cdot2^{k-3} - 25\cdot2^{2k-5}}{3}} m(-q^{2^{k-1} (2^{k-1} + 5)}, q^{3 \cdot 2^{k+1} (2^{k-1} + 3)}, -1) \\
& + 2q^{-35\cdot2^{k-3} - 27\cdot2^{2k-5}} m(-q^{-2^{k-1}(3 \cdot 2^{k-1} + 7)}, q^{3 \cdot 2^{k+1} (2^{k-1} + 3)}, -1) 
\\ & + \theta_{3, 2^{k}}(q^{2^{k-2} + 2}, q^{2^{k-2} + 2}, q)
\end{aligned}
\end{equation*}

and

\begin{equation*} \label{2}
\begin{aligned}
q^{2^{k-1} + 2} f_{3, 2^{k} + 3, 3}&(q^{3\cdot 2^{k-2} + 5}, q^{3\cdot 2^{k-2} + 5}, q) \\
&= -2 q^{-3 \cdot 2^{k-3} (2^{k-2} + 1)}  m(-q^{2^{k-1}(3 \cdot 2^{k-1} + 11)}, q^{3 \cdot 2^{k+1}(2^{k-1} + 3)}, -1) \\
& + 2q^{\frac{-161\cdot2^{k-3} - 121\cdot2^{2k-5}}{3}} m(-q^{-2^{k-1} (5 \cdot 2^{k-1} + 13)}, q^{3 \cdot 2^{k+1}(2^{k-1} + 3)}, -1) \\
& + q^{2^{k-1} + 2} \theta_{3, 2^{k}}(q^{3 \cdot 2^{k-2} + 5}, q^{3 \cdot 2^{k-2} + 5}, q).
\end{aligned}
\end{equation*}

Arguing as in the odd case we obtain

\begin{equation}
\begin{aligned} \label{caseiibis}
\mathcal{R}_2^{(k)}(q) & = -2 q^{\frac{-29\cdot2^{k-3} - 25\cdot2^{2k-5}}{3}} \bigl [ m(q^{-2^k (2^{k-2} +1)}, q^{3\cdot 2^{k-1}(2^{k-1} + 3)}, z) 
+ \xi(q^{-2^k (2^{k-2}+1)}, q^{3\cdot 2^{k-1}(2^{k-1} + 3)}, z) \bigr] \\
& -2q^{-3 \cdot 2^{k-3} (2^{k-2} + 1)} \bigl [ m(q^{2^{k-1}}, q^{3\cdot 2^{k-1}(2^{k-1} + 3)}, z) + \xi(q^{2^{k-1}}, q^{3\cdot 2^{k-1}(2^{k-1} + 3)}, z) \bigr ] \\
& + \frac{\theta_{3, 2^k}(q^{2^{k-2}+2},q^{2^{k-2}+2},q)}{j(q, q^3)} +  \frac{q^{2^{k-1} + 2} \theta_{3, 2^k}(q^{3 \cdot 2^{k-2} + 5}, q^{3 \cdot 2^{k-2} + 5}, q)}{j(q, q^3)}.
\end{aligned}
\end{equation}

Now, for (\ref{main3}), apply Proposition \ref{keyprop} and let $q=q^2$, $c \to \infty$ and $b = q$ in \eqref{limitBailey} to get that

\begin{equation} \label{third}
\begin{aligned}
\mathcal{R}_3^{(k)}(q) & = \frac{(q;q^2)_{\infty}}{(q^2;q^2)_{\infty}}  \sum_{n \geq 0 \atop |j| \leq n} (-1)^{n+j} q^{(2^{k-1}+2)n^2 +(2^{k-1}+1)n - 2^{k-1}j^2}(1+q^{2n+1}) \\
& = \frac{(q;q^2)_{\infty}}{(q^2;q^2)_{\infty}}  \Biggl ( \sum_{n \geq 0} (-1)^n q^{(2^{k-1} + 2)n^2 + (2^{k-1} + 1)n} \sum_{j=-n}^{n} (-1)^j q^{-2^{k-1} j^2}  \\
& \hskip1in + q \sum_{n \geq 0} (-1)^n q^{(2^{k-1} + 2)n^2 + (2^{k-1} + 3)n} \sum_{j=-n}^{n} (-1)^j q^{-2^{k-1} j^2} \Biggr).
\end{aligned}
\end{equation}

Again we replace $n$ with $-n-1$ in the second sum of (\ref{third}), let $n=(r+s)/2$ and $j=(r-s)/2$, and apply (\ref{fprop1}) in the penultimate step to arrive at

\begin{equation*} 
\begin{aligned}
\mathcal{R}_3^{(k)}(q) & =  \frac{(q;q^2)_{\infty}}{(q^2;q^2)_{\infty}} \Biggl ( \Bigl(  \sum_{\substack{r, s \geq 0 \\ r \equiv s \imod{2} }} - \sum_{\substack{r, s < 0 \\ r \equiv {s \imod{2}} }} \Bigr) (-1)^{r-s} q^{\frac{1}{2} r^2 + (2^{k-1} + 1)rs + \frac{2^{k-1} + 1}{2} r + \frac{1}{2} s^2 + \frac{2^{k-1} + 1}{2}s}   \Biggr) \\
& =  \frac{(q;q^2)_{\infty}}{(q^2;q^2)_{\infty}} \Biggl ( f_{1, 2^{k-1} + 1, 1}(-q^{2^{k-1} + 3}, -q^{2^{k-1} + 3}, q^4) \\
&\hskip1in - q^{2^{k} + 3} f_{1, 2^{k-1} + 1, 1}(-q^{3 \cdot 2^{k-1} + 7}, -q^{3 \cdot 2^{k-1} + 7}, q^4) \Biggl) \\
& = \frac{(q;q^2)_{\infty}}{(q^2;q^2)_{\infty}} f_{1,2^{k-1}+1,1}(q^{2^{k-2}+1},-q^{2^{k-2}+1},q). \\
\end{aligned}
\end{equation*}

By Theorem \ref{hm}, (\ref{g}), (\ref{j1}) and (\ref{j2}), we have

\begin{equation*}
\begin{aligned}
& f_{1,2^{k-1}+1,1}(q^{2^{k-2}+1},-q^{2^{k-2}+1},q) \\
& = q^{-2^{k-3}(2^{k-2} + 1)} j(-1,q) m(q^{2^{k-2}},q^{2^{2k-2}+2^k},-1) + \theta_{1, 2^{k-1}}(q^{2^{k-2}+1},-q^{2^{k-2}+1},q)
\end{aligned}
\end{equation*} 

\noindent and so

\begin{equation} \label{caseiii}
\begin{aligned}
\mathcal{R}_3^{(k)}(q) & = 2 q^{-2^{k-3}(2^{k-2} + 1)} m(q^{2^{k-2}},q^{2^{2k-2}+2^k},-1) + \frac{2 \theta_{1, 2^{k-1}}(q^{2^{k-2}+1},-q^{2^{k-2}+1},q)}{j(-1,q)}.
\end{aligned}
\end{equation}

Finally, for (\ref{main4}), applying Proposition \ref{keyprop} and letting $b=\sqrt{q}$ and $c=-\sqrt{q}$ in \eqref{limitBailey}, we have that 

\begin{equation} \label{fourth}
\begin{aligned}
\mathcal{R}_4^{(k)}(q) &= \frac{(q;q^2)_{\infty}}{(q^2;q^2)_{\infty}}  \sum_{n \geq 0 \atop |j| \leq n} (-1)^{n+j} q^{((2^{k-1}+1)n^2 +(2^{k-1}+1)n)/2 - 2^{k-2}j^2} \\
&= \frac{(q;q^2)_{\infty}}{2(q^2;q^2)_{\infty}} \Biggl ( \sum_{n \geq 0 \atop |j| \leq n} (-1)^{n+j} q^{((2^{k-1}+1)n^2 +(2^{k-1}+1)n)/2 - 2^{k-2}j^2} \\
& \hskip1in +  \sum_{n \geq 0 \atop |j| \leq n} (-1)^{n+j} q^{((2^{k-1}+1)n^2 +(2^{k-1}+1)n)/2 - 2^{k-2}j^2} \Biggr). 
\end{aligned}
\end{equation}

Once more replace $n$ with $-n-1$ in the second sum of (\ref{fourth}), let $n=(r+s)/2$ and $j=(r-s)/2$ and apply (\ref{fprop2}) to get

\begin{equation*} 
\begin{aligned}
\mathcal{R}_4^{(k)}(q) &=  \frac{(q;q^2)_{\infty}}{2(q^2;q^2)_{\infty}} \Biggl ( \Bigl(  \sum_{\substack{r, s \geq 0 \\ r \equiv s \imod{2} }} - \sum_{\substack{r, s < 0 \\ r \equiv {s \imod{2}} }} \Bigr) (-1)^{r-s} q^{\frac{1}{8} r^2 + \frac{2^{k} + 1}{4} rs + \frac{2^{k-1} + 1}{4} r + \frac{1}{8} s^2 + \frac{2^{k-1} + 1}{4}s}   \Biggr) \\
& =\frac{(q;q^2)_{\infty}}{2(q^2;q^2)_{\infty}} \Bigl( f_{1,2^k+1,1}(-q^{2^{k-2}+1},-q^{2^{k-2}+1},q) \\
&\hskip1in - q^{2^{k-1}+1}f_{1,2^k+1,1}(-q^{3(2^{k-2}+1)-1},-q^{3(2^{k-2}+1)-1},q) \Bigr) \\
& =  \frac{(q;q^2)_{\infty}}{(q^2;q^2)_{\infty}} f_{1,2^k+1,1}(-q^{2^{k-2}+1},-q^{2^{k-2}+1},q). \\
\end{aligned}
\end{equation*}

By Theorem \ref{hm}, (\ref{g}) and (\ref{j1}), we have

\begin{equation*}
\begin{aligned}
f_{1,2^{k}+1,1}(-q^{2^{k-2}+1},-q^{2^{k-2}+1},q) 
& = 2 q^{-2^{k-3}(2^{k-2} + 1)} j(-1,q) m(-q^{2^{k-1}(2^{k-1} + 1)}, q^{2^{k+1}(2^{k-1} + 1)}, -1) \\ &+ \theta_{1, 2^{k}}(-q^{2^{k-2}+1},-q^{2^{k-2}+1},q)
\end{aligned}
\end{equation*} 

\noindent and so

\begin{equation} \label{caseiv}
\begin{aligned}
\mathcal{R}_4^{(k)}(q) & = 4 q^{-2^{k-3}(2^{k-2} + 1)} m(-q^{2^{k-1}(2^{k-1} + 1)}, q^{2^{k+1}(2^{k-1} + 1)}, -1) + \frac{2\theta_{1, 2^{k}}(-q^{2^{k-2}+1},-q^{2^{k-2}+1},q)}{j(-1,q)}.
\end{aligned}
\end{equation} 
\end{proof}

\section{Proof of Corollary \ref{mockid}}

To prove Corollary \ref{mockid} we will compare the Appell-Lerch sums $m(x,q,z)$ appearing in \eqref{casei}, \eqref{caseii}, and \eqref{caseiv} to the Appell-Lerch sums corresponding to the classical $q$-hypergeometric mock theta functions, as recorded in \cite{Hi-Mo1}.

\begin{proof}[Proof of Corollary \ref{mockid}]

We begin with (\ref{mockid1}).  Taking $k=3$ and $z=q^3$ in (\ref{casei}), we find that
$$
\mathcal{R}_1^{(3)}(-q) = 2q^{-1}\left[m(q^2,q^{12},-q^3) + \xi(q^2,q^{12},-q^3)\right] + \frac{\theta_{1,4}(-q^3,-q^3,q^2)}{j(-q,q^2)} - q^5\frac{\theta_{1,4}(-q^9,-q^9,q^2)}{j(-q,q^2)}.
$$
On the other hand, equation (4.9) of \cite{Hi-Mo1} says that
$$
\nu(q) = 2q^{-1}m(q^2,q^{12},-q^3) + \frac{J_1J_{3,12}}{J_2}.
$$
Thus the claim is equivalent to the identity
$$
2q^{-1} \xi(q^2, q^{12}, -q^3) + \frac{\theta_{1,4}(-q^3, -q^3, q^2)}{j(-q,q^2)} - \frac{q^5 \theta_{1,4}(-q^9, -q^9, q^2)}{j(-q, q^2)} = \frac{J_1 J_{3, 12}}{J_2}.
$$
This is a routine identity involving modular forms and functions and hence may be verified with a finite computation.   We carried this out using
F.G. Garvan's computer package (see \url{http://www.math.ufl.edu/~fgarvan/qmaple/theta-supplement/}).

Next, for (\ref{mockid2}) we take $k=4$ and $z=q^8$ in (\ref{casei}) to obtain 
$$
\mathcal{R}_1^{(4)}(q) = 2q^{-4}\left[m(q^4,q^{40},q^8) + \xi(q^4,q^{40},q^8)\right] + \frac{\theta_{1,8}(q^5,q^5,q^2)}{j(q,q^2)} + q^9\frac{\theta_{1,8}(q^{15},q^{15},q^2)}{j(q,q^2)}.
$$
On the other hand, equation (4.43) in \cite{Hi-Mo1} with $q=q^4$ reads
$$
\phi(q^4) = -2q^{-4}m(q^4,q^{40},q^8) + \frac{J_5J_{10}J_{4,10}}{J_{2,5}J_{2,10}}\Bigg |_{q=q^4}.
$$
Comparing these two equations establishes \eqref{mockid2} (and also provides an expression for $M_1(q)$).

Equations \eqref{mockid3} and \eqref{mockid5} are quite similar, so we just mention that \eqref{mockid3} follows from taking $k=3$ and $z=q^{12}$ in (\ref{caseii}) and comparing with equation (4.33) in \cite{Hi-Mo1}, while \eqref{mockid5} follows upon comparing \eqref{caseiv} with the case $q \to q^{2^{k-1}(2^{k-1} + 1)}$ of equation (4.3) in \cite{Hi-Mo1}.

\end{proof}

\section{Concluding Remarks}
We have described one way to use the Bailey machinery to produce families of $q$-hypergeometric multisums which are mock theta functions.   It would be interesting to find others. For example, one could study other change of bases found in \cite{bis}. It would also be interesting to establish simpler expressions for the modular forms $M_1(q)$, $M_2(q)$, and $M_3^{(k)}(q)$ occurring in Corollary \ref{mockid}. Although we did not address it here, it would also be natural to study asymptotics and congruences for the coefficients of $\mathcal{R}_i^{(k)}(q)$. Finally, it is mentioned in Section $8$ of \cite{St1} that some preliminary work has been done toward understanding the combinatorial significance of certain modular $q$-hypergeomatric multisums constructed using change of base lemmas for Bailey pairs.  These modular multsums resemble the $\mathcal{R}_i^{(k)}(q)$ and it would be nice to see this work carried out.

\section*{Acknowledgements} The second author would like to thank the Institut des Hautes {\'E}tudes
Scientifiques for their support during the preparation of this paper and Don Zagier for his comments. The second author was partially funded by Science Foundation Ireland 08/RFP/MTH1081.

\end{document}